\documentclass[12pt]{amsart}
\usepackage[left=0.8in,right=0.8in,top=1.0in,bottom=1.0in]{geometry}
\usepackage[T1]{fontenc}
\usepackage{lmodern}
\usepackage{amsmath,amssymb,amsthm,mathtools}
\usepackage{microtype}
\usepackage{enumitem}
\usepackage[dvipsnames]{xcolor}
\usepackage{hyperref}
\hypersetup{
    colorlinks=true,
    citecolor=blue,
    linkcolor=purple,
    filecolor=magenta,
    urlcolor=cyan,
    breaklinks=true
    }
\usepackage[nameinlink]{cleveref}
\usepackage{url}
\usepackage{thmtools}
\usepackage{booktabs}

\theoremstyle{plain}
\declaretheorem[name=Theorem]{theorem}
\declaretheorem[name=Corollary, sibling=theorem]{corollary}
\declaretheorem[name=Lemma, sibling=theorem]{lemma}

\declaretheorem[name=Example, sibling=theorem]{example}
\declaretheorem[name=Proposition, sibling=theorem]{proposition}

\newcommand{\ghost}[1]{}
\newcommand{\R}{\mathbb{R}}
\newcommand{\F}{\mathbb{F}}
\newcommand{\N}{\mathbb{N}}

\newcommand{\cG}{\mathcal{G}}
\newcommand{\w}{\mathbf{w}}
\newcommand{\hlam}{\widehat{\lambda}}
\newcommand{\norm}[1]{\left\lVert #1\right\rVert}
\newcommand{\ip}[2]{\left\langle #1,#2\right\rangle}
\renewcommand{\emptyset}{\varnothing}
\DeclareMathOperator{\prank}{partition\text{-}rank}
\DeclareMathOperator{\srank}{slice\text{-}rank}
\DeclareMathOperator{\frank}{frank}

\title[Multiple Distance Ramsey Bounds For Graphs in Euclidean Spaces]{Multiple Distance Ramsey Bounds For Graphs in Euclidean Spaces}
\date{\today}

\author[Ay\c{s}eg\"{u}l Kula]{Ay\c{s}eg\"{u}l Kula}
\address{Pomona College, Claremont, California, United States of America}
\email{akam2024@mymail.pomona.edu}

\author[Mohamed Omar]{Mohamed Omar}
\address{York University, Toronto, Ontario, Canada}
\email{omarmo@yorku.ca}

\author[Jonah Stockwell]{Jonah Stockwell}
\address{Columbia University, New York, New York, United States of America}
\email{js5384@columbia.edu}

\author[Mckinley Xie]{Mckinley Xie}
\address{Texas A\&M University, College Station, Texas, United States of America}
\email{mckinleyxie@tamu.edu}


\begin{document}

\begin{abstract}
For a finite set $A \subset \R_{>0}$ and a finite graph $H$, let $\chi_H(\R^n;A)$ be the minimum number of colors required to color $\R^n$ while avoiding a monochromatic copy of $H$ whose edges have distances in $A$. Extending the graph-copy framework of Axenovich, Liu, and Sagdeev and a multiple distance theorem of Naslund, we prove for any positive integer $m$,
\[
\chi_H(\R^n;m):=\max_{\substack{A \subseteq \R_{>0} \\ |A|=m}} \chi_H(\R^n;A) \geq
\left(\Gamma_{\chi}\sqrt{\frac{m+1}{\Xi(H)}}+o(1)\right)^n.
\]
Here, $\Gamma_{\chi}$ is a constant and $\Xi(H)$ is an explicit structural parameter that can be substantially smaller than $|V(H)|-1$, thereby recovering Naslund's similar bound for complete graphs and improving the general bound inherited from the corresponding clique for many graph families. Along the way, we construct a weighted strengthening of the semi-diagonal flattening rank theorem of Correia, Sudakov, and Tomon.
\end{abstract}

\maketitle

\section{Introduction}

A central problem in geometric combinatorics, the Hadwiger-Nelson problem, asks for the chromatic number of the graph on $\R^2$ in which two points are adjacent when they are at Euclidean distance one. Its difficulty is a persistent reminder that even elementary distance constraints can hide substantial structure. The known bounds in the plane are
\[
5\leq \chi(\R^2)\leq 7.
\]
The upper bound comes from a classical hexagonal coloring, while the lower bound was recently raised from four to five by de Grey, with an independent proof by Exoo and Ismailescu \cite{soifer2009,degrey2018,exoo2020}. In higher dimensions, the same question connects to the Frankl-Wilson method, the Larman-Rogers upper bound, and Raigorodskii's exponential lower bounds \cite{larmanrogers1972,franklwilson1981,raigorodskii2000}. The combination of elementary formulation, stubborn planar mystery, and high-dimensional algebraic structure is one reason distance coloring problems have remained so influential.

Erd\H{o}s was one of the principal architects of the broader Euclidean Ramsey viewpoint. In the papers of Erd\H{o}s, Graham, Montgomery, Rothschild, Spencer, and Straus, the unit distance problem was placed into a larger program: color Euclidean space and force monochromatic copies of prescribed finite configurations \cite{erdos1973i,erdos1975ii,erdos1975iii}. That program led to the study of Ramsey sets, simplices, boxes, products, and many related geometric configurations, and it continues to guide modern work in the subject \cite{franklrod1990,graham2004}. The underlying theme is not merely that some finite pattern must appear. It is that the geometry and combinatorics of the pattern should be reflected in the number of colors required to avoid it.

Naslund's theorem on multiple forbidden distances fits this theme in a particularly sharp algebraic form \cite{naslund2022chromatic}. Instead of forbidding only distance one, one fixes $m$ distances and asks how many colors are needed to avoid a monochromatic clique $K_h$ on $h$ vertices with all of its edge lengths in that distance set. Setting $\chi_{K_h}(\R^n;m)$ to be the maximum of this over all $m$-element distance sets, Naslund proves that as $n \to \infty$,
\[
\chi_{K_h}(\R^n;m)\geq
\left(\Gamma_{\chi}\sqrt{\frac{m+1}{h-1}}+o(1)\right)^n,
\]
where
\[
\Gamma_{\chi}=\sqrt{\frac{\pi}{2}}\max_{x>0}\frac{1-e^{-x}}{\sqrt{x}}
=0.7998308498\ldots .
\]
The proof belongs to the polynomial method lineage that began with the work of Croot, Lev, and Pach and Ellenberg and Gijswijt, was clarified through Tao's slice rank formulation, and was extended by Naslund through partition rank \cite{crootlevpach,capset,tao2016capset,naslund2020partition}. For comparison, the Larman-Rogers argument gives the general upper bound
\[
\chi_{K_2}(\R^n;m)\leq (3+o_n(1))^{mn},
\]
while, for the particular distance set $\{1,\sqrt{2},\ldots,\sqrt{m}\}$ used in Naslund's construction, Kupavskii's bound is
\[
\chi_{K_2}(\R^n;\{1,\sqrt{2},\ldots,\sqrt{m}\})
\leq \left(2(\sqrt{m}+1)+o_n(1)\right)^n
\]
\cite{larmanrogers1972,kupavskii2010chromatic,naslund2022chromatic}. Thus the known lower and upper bounds for this special distance set have the same square root order in $m$, though their constants remain different.

A different recent development changes the forbidden object rather than the set of distances. Axenovich, Liu, and Sagdeev introduced a graph-copy version of Euclidean Ramsey theory in which one prescribes only the unit distances encoded by a graph $H$ \cite{axenovich2025ramsey}. This is not a cosmetic variation of the Hadwiger-Nelson problem. It replaces a rigid configuration by a flexible set of edge constraints: edges of $H$ must be realized by unit distances, while nonedges impose no restrictions. Their work shows that this graph-copy viewpoint has its own structure, its own product methods, and its own difficult small cases, while still extending both the unit distance chromatic number and classical Euclidean Ramsey questions.

The purpose of this article is to merge these two directions. Let $A\subset\R_{>0}$ be finite. Write $G_A(\R^n)$ for the graph with vertex set $\R^n$ in which two distinct points are adjacent if their distance lies in $A$. Let $H$ be a finite simple graph with at least one edge. A copy of $H$ in $G_A(\R^n)$ is an injective map
\[
\phi:V(H)\longrightarrow \R^n
\]
such that
\[
\norm{\phi(u)-\phi(v)}_2\in A
\]
for every edge $uv\in E(H)$. There is no condition on a nonedge of $H$. We write $\chi_H(\R^n;A)$ for the least number of colors needed to color $\R^n$ with no monochromatic copy of $H$ in $G_A(\R^n)$, and for a positive integer $m$ we put
\[
\chi_H(\R^n;m)=
\max_{\substack{A\subset\R_{>0}\\ |A|=m}}\chi_H(\R^n;A).
\]
When $A=\{1\}$, this is the graph-copy problem of Axenovich, Liu, and Sagdeev. When $H$ is a clique, it is Naslund's multiple distance problem.

The immediate way to apply Naslund's theorem to a graph $H$ is to embed $H$ into the complete graph on the same vertex set. If $|V(H)|=h$, then every monochromatic copy of $K_h$ contains a monochromatic copy of $H$, and therefore
\[
\chi_H(\R^n;m)\geq
\left(\Gamma_{\chi}\sqrt{\frac{m+1}{h-1}}+o(1)\right)^n.
\]
Our main result replaces the clique denominator $h-1$ by a statistic of the graph $H$ that is potentially much smaller. Recall that if $H$ is a graph and $\pi$ is a partition of $V(H)$ into at least two nonempty blocks, the quotient graph $H/\pi$ is the simplification of the graph whose vertices are the blocks of $\pi$, with two different blocks adjacent in $H/\pi$ when at least one edge of $H$ has one endpoint in each block. 

To define our statistic governing the base of the exponential lower bound on $\chi_H(\R^n;m)$, first define 
\[
\lambda(H)=\max_{\substack{\pi\text{ a partition of }V(H)\\ |\pi|\geq 2}}\delta(H/\pi),
\]
where $\delta(Q)$ is the minimum degree of a graph $Q$ and set $\lambda(K_1)=0$. The definition needs one adjustment for disconnected graphs. Suppose that the connected components of $H$ are $C_1,\ldots,C_s$. A connected envelope of $H$ is a finite connected graph $F$ that contains each $C_i$ as a subgraph; the embeddings of the different components into $F$ are not required to be disjoint. Define
\[
\hlam(H)=\min_F\lambda(F),
\]
where the minimum is over all connected envelopes $F$ of $H$. When $H$ is connected, we shall prove that $\hlam(H)=\lambda(H)$. Finally, arrange the vertex degrees of $H$ as
\[
d_1(H)\geq d_2(H)\geq\cdots\geq d_h(H).
\]
Because $H$ has an edge, it has at least two vertices and $d_2(H)$ is defined. Set
\begin{equation}\label{eq:def-xi}
\Xi(H)=\min\{\hlam(H),d_2(H)\}.
\end{equation}

Our main theorem replaces $h-1$ in Naslund's result with $\Xi(H)$:

\begin{theorem}\label{thm:main}
Let $H$ be a finite simple graph with at least one edge, and let $m$ be a positive integer. Then, as $n\to\infty$,
\[
\chi_H(\R^n;m) \geq
\left(\Gamma_{\chi} \sqrt{\frac{m+1}{\Xi(H)}}+o(1)\right)^n.
\]
\end{theorem}

The theorem recovers Naslund's result for $H=K_h$, since $\Xi(K_h)=h-1$. One can of course inherit Naslund's result and use $h-1$ in place of $\Xi(H)$ for any graph on $h$ vertices. However, its force is visible for sparse graph families. Stars and matchings have $\Xi=1$. Long paths and cycles have $\Xi=2$, even though their quotient parameter $\lambda$ grows on the order of the square root of their order. Every family of bounded maximum degree has bounded $\Xi$. 

The paper is organized as follows. \Cref{section:prelim} develops the rank, partition lattice, quotient graph, and finite Euclidean preliminaries in the order in which they are used. \Cref{section:weighted-frank} proves a weighted flattening rank theorem using the rank deletion step from the proof of Correia, Sudakov, and Tomon. \Cref{section:finite-rank} constructs the graph-copy tensor and gives the two color class estimates that produce $d_2(H)$ and $\lambda(H)$. \Cref{section:transfer} explains the passage from those finite estimates to Euclidean colorings and proves \Cref{thm:main}. \Cref{section:the-statistic} develops $\Xi$ in detail, with special attention to graph families of growing order. \Cref{section:conclusion} discusses future directions.

\section{Preliminaries}\label{section:prelim}

\subsection{Tensors and three rank notions}

Let $k\geq2$. Let $X_1,\ldots,X_k$ be finite sets and let $\F$ be a field. A tensor is simply a function
\[
T:X_1\times\cdots\times X_k\longrightarrow\F.
\]
It may be helpful to think of $T$ as a $k$-dimensional array whose entries are indexed by one element from each set. The rank notions below measure how efficiently this array can be written as a sum of functions that split into independent groups of variables.

A slice is a tensor of the form
\[
T(x_1,\ldots,x_k)=f(x_i)g(x_1,\ldots,x_{i-1},x_{i+1},\ldots,x_k)
\]
for some coordinate $i$, $f:X_i \to \F$ and $g:X_1 \times \cdots \times X_{i-1} \times X_{i+1} \times \cdots \times X_k \to \F$. The slice rank of $T$, denoted $\srank(T)$, is the least number of slices whose sum is $T$. Thus one factor of a slice sees a single variable and the other factor sees all remaining variables. Partition rank permits more general separations. A nontrivial partition of $[k]=\{1,\ldots,k\}$ is a partition with at least two blocks. If $\rho=\rho_1|\cdots|\rho_r$ is such a partition, then $T$ has partition rank one with respect to $\rho$ if
\[
T(x_1,\ldots,x_k)=\prod_{j=1}^r T_j((x_i)_{i\in\rho_j}),
\]
where for each $j$, $T_j:\prod_{i \in \rho_j} X_i \to \F$. The partition rank of $T$, denoted $\prank(T)$, is the least number of partition rank one tensors whose sum is $T$. Every slice has partition rank one, but partition rank also allows the variables to split into several blocks of size larger than one. If $B$ is finite and $T:B^k\to\F$ is nonzero exactly when $b_1=\cdots=b_k$, then
\begin{equation}\label{eq:diagonal-rank}
\prank(T)=|B|;
\end{equation}
this is the partition rank analogue of the fact that a diagonal matrix with nonzero diagonal entries has full matrix rank \cite{naslund2020partition,tao2016capset}. We use only \labelcref{eq:diagonal-rank}; no other structural result about partition rank is required.

The third notion is flattening rank. If $T:X_1\times\cdots\times X_k\to\F$, its $i$-flattening matrix has rows indexed by $X_i$, columns indexed by the product of all other $X_j$, and the entry in row $x_i$ and column $(x_j)_{j\neq i}$ equal to $T(x_1,\ldots,x_k)$. We write
\[
\frank_i(T)=\operatorname{rank}M_i(T).
\]
Equivalently, $\frank_i(T)$ is the least number of terms needed to express $T$ as
\[
T=\sum_{r=1}^R f_r(x_i)g_r((x_j)_{j\neq i}).
\]
This equivalent formulation is the one used in the polynomial estimates.

A tensor $T:B^k\to\F$ is \emph{semi-diagonal} if
\[
T(b,\ldots,b)\neq 0\qquad\text{for every }b\in B
\]
and
\[
T(b_1,\ldots,b_k)=0\qquad\text{whenever }b_1,\ldots,b_k
\text{ are all distinct}.
\]
There is no restriction on a tuple that is neither fully diagonal nor all distinct. Correia, Sudakov, and Tomon proved that a semi-diagonal tensor satisfies
\[
\sum_{i=1}^k\frank_i(T)\geq \frac{k}{k-1}|B|
\]
\cite{correia2021flattening}. In \Cref{section:weighted-frank} we prove a weighted strengthening that allows one specified flattening rank to be omitted entirely.

\subsection{The partition lattice and the distinctness indicator}

For a finite set $U$, let $\Pi_U$ be the set of partitions of $U$, ordered by refinement. For $U=[k]$, the minimal partition is
\[
\hat 0=1|2|\cdots|k,
\]
and the maximal partition is
\[
\hat 1=12\cdots k.
\]
For a subset $P\subseteq[k]$, write $x_P=(x_i)_{i\in P}$, and let $\delta_P(x_P)$ be one when all coordinates indexed by $P$ are equal and zero otherwise. If $\pi\in\Pi_{[k]}$, define
\[
\delta_\pi(x_1,\ldots,x_k)=1
\]
when $x_i=x_j$ whenever $i$ and $j$ lie in the same block of $\pi$, and define it to be zero otherwise. Thus $\delta_\pi$ forces equality within each block, but it does not force different blocks to receive different values.

Let $\mu$ be the M\"obius function of the partition lattice \cite{stanley2012}. Naslund's distinctness indicator for $k \geq 2$ is
\begin{equation}\label{eq:ik}
I_k(x_1,\ldots,x_k)=
\sum_{\substack{\pi\in\Pi_{[k]}\\ \pi\neq\hat 1}}
\mu(\hat 0,\pi)\delta_\pi(x_1,\ldots,x_k).
\end{equation}
The omission of the one-block partition is essential: it makes the function nonzero on the full diagonal.

\begin{lemma}\label{lem:distinctness}
For every $(x_1,\ldots,x_k)$,
\[
I_k(x_1,\ldots,x_k)=
\begin{cases}
1, & x_1,\ldots,x_k\text{ are all distinct},\\
(-1)^k(k-1)!, & x_1=\cdots=x_k,\\
0, & \text{otherwise}.
\end{cases}
\]
In particular, the full diagonal value is nonzero over every field of characteristic greater than $k$.
\end{lemma}

This can be established by a M\"{o}bius inversion calculation that is a special case of the partition-indicator viewpoint developed in \cite{omar2025partition}. In this paper, it will be multiplied by a graph-copy tensor that vanishes on all-distinct forbidden tuples. The product will then vanish off the full diagonal and retain the nonzero value from \Cref{lem:distinctness} on the diagonal.

\subsection{Quotient graphs as equality patterns}

The quotient graph defined in the introduction is exactly what remains from a graph-copy tensor after the equalities in $\delta_\pi$ are imposed. Suppose that the blocks of $\pi$ are $P_1,\ldots,P_r$, and let $y_j$ denote the common value of all variables $x_i$ with $i\in P_j$. An edge of $H$ whose endpoints lie in the same block produces a diagonal factor. Several edges running between the same two blocks produce repeated copies of the same factor. In our finite-field model that factor takes only the values zero and one, so repeated copies collapse to one. The surviving factors are therefore indexed by the edges of the simple quotient $H/\pi$.

\begin{example}\label{ex:p4-quotient}
Let $H=P_4$ have edges $12,23,34$. Take
\[
\pi=14|2|3.
\]
The block $\{1,4\}$ is adjacent in $H/\pi$ to the block $\{2\}$ because of the edge $12$, and it is adjacent to $\{3\}$ because of the edge $34$. The remaining edge $23$ joins the other two blocks. Thus
\[
P_4/\pi\cong K_3.
\]
This small example is important conceptually: even a path can have a quotient denser than the original graph. The parameter $\lambda$ is designed to record the largest minimum degree created by any such equality pattern.
\end{example}

\subsection{The distance detector}

In Naslund's lower bound argument for $\chi_{K_h}(\R^n;m)$, a special subset $S \subset \{0,1,\ldots,\ell\}^n$ is constructed that has sufficiently large size, and the lower bound is a strategic multiple of it. We present a different set $S$ that gets the same job done but substantially simplifies Naslund's argument. We now introduce this set $S$.

\begin{lemma} \label{lem:size_of_S}
Let $d_n \in \N$ be positive with $d_n=O(\sqrt{n})$, and choose $\ell=2d_n$. There exist a set $S\subseteq\{0,\ldots,\ell\}^n$ and a $c\in\{0,\ldots,\ell\}^n$ such that $\norm{x-c}_2^2=\norm{y-c}_2^2$ for every $x,y\in S$, the diameter of $S$ is at most $d_n$, and
\[
|S| \geq \left(\left(d_n\sqrt{\frac{e\pi}{2n}}\right)(1+o(1))\right)^n.
\]
\end{lemma}

\begin{proof}
We will construct a set with the required size.

Let $c=\ell/2\cdot\mathbf{1}=d_n\mathbf{1}$ where $\mathbf{1}$ is the all-ones vector. For $b\in[-1/2,1/2)^n$ let $S^\prime(b)$ be the set of lattice points contained in the $n$-ball of diameter $d_n$ centered at $c+b$. Since $d_n\geq1$, this ball is contained in $[0,\ell]^n$.

Let $f(b):=|S^\prime(b)|$. Note that the integral of $f$ over $[-1/2,1/2)^n$ is exactly the volume of an $n$-dimensional ball of diameter $d_n$; one can think of $f$ as the result of cutting up the $n$-dimensional ball into unit hypercubes and ``stacking'' them all into the region $[-1/2,1/2)^n$. As this region has volume 1, there exists a particular $b_0\in[-1/2,1/2)^n$ such that
\[
|S^\prime(b_0)|=f(b_0)\geq
\left(\frac{d_n}{2}\right)^nV_n
=\left(\frac{d_n}{2}\right)^n\frac{\pi^{n/2}}{\Gamma(n/2+1)},
\]
where $V_n$ is the volume of the unit $n$-ball and $\Gamma$ is the classical gamma function.

Using the upper bound $\Gamma(x)\leq\mathrm{exp}(\int_2^{x+1}\ln t\,\mathrm{d}t)=\frac{e^2}{4}\left(\frac{x+1}{e}\right)^{x+1}$, we get
\[
|S^\prime(b_0)|\geq
\frac{16}{(n+4)^2}
\left(d_n\sqrt{\frac{e\pi}{2(n+4)}}\right)^n.
\]

For every $x\in S^\prime(b_0)$, the triangle inequality gives $\norm{x-c}_2\leq(d_n+\sqrt n)/2$. Since $x$ and $c$ are lattice points, $\norm{x-c}_2^2$ is an integer and therefore takes on at most $\lfloor((d_n+\sqrt n)/2)^2\rfloor+1\leq((d_n+\sqrt n)^2+4)/4$ distinct values. By the pigeonhole principle, one of the corresponding level sets $S$ has size at least
\[
|S|\geq
\frac{64}{(n+4)^2\left((d_n+\sqrt n)^2+4\right)}
\left(d_n\sqrt{\frac{e\pi}{2(n+4)}}\right)^n.
\]
This level set has constant squared distance from $c$, and its diameter is at most $d_n$ because it is contained in a ball of diameter $d_n$. Since $d_n=O(\sqrt n)$, the $n$th root of the factor outside the power tends to one, and replacing $n+4$ by $n$ inside the power also introduces a factor $1+o(1)$. This yields the desired asymptotic bound.
\end{proof}

Consequently, for $x,y\in S$ (where $S$ is the set from Lemma~\ref{lem:size_of_S}) and some constant $R$,
\begin{equation}\label{eq:half-distance}
\frac12 \norm{x - y}_2^2 = \frac12\norm{(x-c)-(y-c)}_2^2=R-\ip{x-c}{y-c},
\end{equation}
which is an integer. The constant norm condition is what changes the distance expression from a quadratic polynomial in one point to a linear polynomial in that point once the other point is fixed.

Let $p$ be an odd prime and work over $\F_p$. Define
\begin{equation}\label{eq:detector}
D_p(x,y)=1-\left(\frac12\norm{x-y}_2^2\right)^{p-1}.
\end{equation}
Fermat's little theorem gives
\[
D_p(x,y)=
\begin{cases}
1, & \frac12\norm{x-y}_2^2\equiv 0\pmod p,\\
0, & \frac12\norm{x-y}_2^2\not\equiv 0\pmod p.
\end{cases}
\]
In particular, $D_p$ takes only the values zero and one on the finite model, and therefore $D_p(x,y)^r=D_p(x,y)$ for every positive integer $r$.

Put
\[
d_{\max}(S)=\max_{x,y\in S}\frac12\norm{x-y}_2^2.
\]
If $d_{\max}(S)<(m+1)p$, then
\[
\frac12\norm{x-y}_2^2\equiv0\pmod p
\]
is equivalent to
\[
\frac12\norm{x-y}_2^2\in\{0,p,2p,\ldots,mp\}.
\]
Thus, for distinct points, $D_p$ detects precisely the distance set
\begin{equation}\label{eq:amp}
A_{m,p}=\{\sqrt{2p},\sqrt{4p},\ldots,\sqrt{2mp}\}.
\end{equation}

Finally, for a real number $L\geq0$, define
\[
M_{\ell,n}(L)=
\left\{\alpha\in\{0,1,\ldots,\ell\}^n:
\alpha_1+\cdots+\alpha_n\leq L\right\}.
\]
The next elementary observation converts a degree bound into a flattening or partition rank bound.

\begin{lemma}\label{lem:monomial-count}
Let $P(z_1,\ldots,z_n)$ be a polynomial whose restriction to $\{0,1,\ldots,\ell\}^n$ has a representative of total degree at most $L$. Then its restriction to any subset of the box belongs to a vector space of dimension at most $|M_{\ell,n}(L)|$.
\end{lemma}

\begin{proof}
For each coordinate $z_j$, divide by
\[
\prod_{r=0}^{\ell}(z_j-r).
\]
The remainder has degree at most $\ell$ in $z_j$, agrees with the original polynomial on every point of the box, and does not have larger total degree. Repeating this reduction in every coordinate gives a representative spanned by the monomials
\[
z^\alpha=z_1^{\alpha_1}\cdots z_n^{\alpha_n}
\]
with $0\leq\alpha_j\leq\ell$ and $|\alpha|\leq L$. There are exactly $|M_{\ell,n}(L)|$ such monomials. Restricting the functions to a smaller set cannot increase the dimension of their span.
\end{proof}

\section{Weighted flattening rank}\label{section:weighted-frank}

The ordinary semi-diagonal theorem controls the sum of all flattening ranks. For the graph-copy tensor, however, one vertex may have much larger degree than all of the others. The correct estimate should be allowed to ignore the flattening belonging to that vertex. We obtain this from a weighted version of the theorem of Correia, Sudakov, and Tomon. The proof rests on the following rank deletion step. It is implicit in the proof of \cite[Theorem~2]{correia2021flattening}. We include its proof for completeness.

\begin{lemma}\label{lem:rank deletion}
Let $k\geq2$. Let $T:B^k\to\F$ be semi-diagonal and suppose that $|B|\geq k$. There exist a set $J\subseteq[k]$ with $|J|\geq2$ and a nonempty set $Y\subseteq B$ with
\[
|Y|\leq |J|-1
\]
such that, for $T'=T|_{(B\setminus Y)^k}$,
\[
\frank_j(T)\geq\frank_j(T')+1
\qquad\text{for every }j\in J. 
\]
\end{lemma}

\begin{proof}
For $a=(a_1,\ldots,a_k)\in B^k$, let $\{a\}$ be the set of values appearing among its coordinates, and let $\phi(a)$ be the set of indices whose coordinate value appears at least twice. For $i\in[k]$, let $a[i]\in\F^B$ be the row vector defined by
\[
a[i](b)=T(a_1,\ldots,a_{i-1},b,a_{i+1},\ldots,a_k).
\]
Thus $\frank_i(T)$ is the dimension of the span of the vectors $a[i]$.

Choose $a\in B^k$ with $T(a)\neq0$ such that $|\{a\}|$ is as large as possible. Since $T$ is semi-diagonal, $|\{a\}|\leq k-1$. Let $\mathcal A$ be the set of all $b\in B^k$ such that $T(b)\neq0$ and $\{b\}=\{a\}$. Join two elements of $\mathcal A$ when they differ in exactly one coordinate, and let $\mathcal C$ be the connected component containing $a$.

Call an index $j\in[k]$ good if $j\in\phi(b)$ for some $b\in\mathcal C$, and let $J$ be the set of good indices. Since $a$ has at most $k-1$ distinct coordinate values, at least two of its coordinates are repeated, so $|J|\geq2$.

If $j\notin J$, then $a_j=b_j$ for every $b\in\mathcal C$. Indeed, whenever two consecutive elements of $\mathcal C$ differ in coordinate $r$, the old value in that coordinate must occur in another coordinate, since the set of coordinate values remains unchanged. Hence $r$ is good, and no coordinate outside $J$ can change along a path in $\mathcal C$.

Set $X=\{a_j:j\notin J\}$ and $Y=\{a\}\setminus X$. The values in $X$ are distinct and occur only once in $a$. Consequently,
\[
|Y|=|\{a\}|-(k-|J|)\leq |J|-1.
\]
Also, $Y$ is nonempty because the indices in $J$ correspond to repeated coordinate values.

Fix $j\in J$. Choose $b\in\mathcal C$ with $j\in\phi(b)$ and put $v_j=b[j]$. The maximality of $|\{a\}|$ implies that $\operatorname{supp}(v_j)\subseteq\{a\}$. Otherwise, replacing the $j$th coordinate of $b$ by an element of $\operatorname{supp}(v_j)\setminus\{a\}$ would produce a nonzero tuple with more distinct coordinate values.

In fact, $\operatorname{supp}(v_j)\subseteq Y$. Suppose that $c\in X\cap\operatorname{supp}(v_j)$. There is a unique $r\notin J$ with $a_r=c$, and this coordinate is fixed throughout $\mathcal C$. Since $j\in\phi(b)$, we have $b_j\neq c$, for otherwise $r$ would also be good. Replacing $b_j$ by $c$ therefore gives a tuple in $\mathcal C$ in which the coordinate $r$ is repeated, contradicting $r\notin J$.

Let $B'=B\setminus Y$ and let $T'=T|_{(B')^k}$. For each $j\in J$, choose rows of $T'$ whose restrictions to $(B')^{k-1}$ form a basis of the $j$th flattening row space of $T'$. Their number is $\frank_j(T')$. The vector $v_j$ is nonzero because $v_j(b_j)=T(b)\neq0$, and its support is contained in $Y$. It is therefore linearly independent from those rows: after restricting to $(B')^{k-1}$, the coefficients of the rows of $T'$ must vanish, and then the coefficient of $v_j$ must vanish as well. Hence
\[
\frank_j(T)\geq\frank_j(T')+1.
\]
This holds for every $j\in J$, completing the proof.
\end{proof}

For nonnegative weights $w_1,\ldots,w_k$, define
\begin{equation}\label{eq:weight-constant}
\w(w_1,\ldots,w_k)=
\min_{\substack{J\subseteq[k]\\ |J|\geq2}}
\frac{\sum_{j\in J}w_j}{|J|-1}.
\end{equation}
This statistic affords a weighted version of the theorem of Correia, Sudakov, and Tomon \cite{correia2021flattening}:

\begin{theorem}[Weighted Flattening Rank Theorem]\label{thm:weighted}
Let $k\geq2$. Suppose $T:B^k\to\F$ is semi-diagonal and $w_1,\ldots,w_k\geq0$. Then
\[
\sum_{i=1}^k w_i\frank_i(T)
\geq \w(w_1,\ldots,w_k)|B|.
\]
\end{theorem}

\begin{proof}
Write
\[
R_w(T)=\sum_{i=1}^k w_i\frank_i(T)
\]
and argue by induction on $|B|$. If $B=\emptyset$, the assertion is immediate. If $1\leq|B|\leq k-1$, then every flattening rank is at least one because $T$ is nonzero on the diagonal. Hence
\[
R_w(T)\geq\sum_{i=1}^k w_i.
\]
Taking $J=[k]$ in \labelcref{eq:weight-constant} gives
\[
\sum_{i=1}^k w_i\geq (k-1)\w(w_1,\ldots,w_k)
\geq |B|\w(w_1,\ldots,w_k).
\]

Now suppose $|B|\geq k$. Apply \Cref{lem:rank deletion}, and write $T'=T|_{(B\setminus Y)^k}$. Flattening rank cannot increase under restriction, and for $j\in J$ it drops by at least one. Therefore
\begin{align*}
R_w(T)
&\geq R_w(T')+\sum_{j\in J}w_j\\
&\geq \w(w_1,\ldots,w_k)(|B|-|Y|)
   +\w(w_1,\ldots,w_k)(|J|-1)\\
&\geq \w(w_1,\ldots,w_k)(|B|-|Y|)
   +\w(w_1,\ldots,w_k)|Y|\\
&=\w(w_1,\ldots,w_k)|B|.
\end{align*}
The second line uses the inductive hypothesis and the definition of $\w$; the third uses $|Y|\leq|J|-1$.
\end{proof}

Setting all weights equal to one gives the theorem of Correia, Sudakov, and Tomon. The specialization needed here is stronger than the maximum-flattening rank consequence.

\begin{corollary}\label{cor:frank}
Let $k\geq2$ and suppose $T:B^k\to\F$ is semi-diagonal. For every $s\in[k]$,
\[
\sum_{i\neq s}\frank_i(T)\geq |B|.
\]
\end{corollary}

\begin{proof}
Set $w_s=0$ and $w_i=1$ for $i\neq s$. For every $J\subseteq[k]$ with $|J|\geq2$,
\[
\sum_{j\in J}w_j\geq |J|-1.
\]
Thus $\w(w_1,\ldots,w_k)\geq1$, and \Cref{thm:weighted} gives the result.
\end{proof}

\section{The finite graph-copy rank estimates}\label{section:finite-rank}

Fix a graph $H$ on vertex set $[h]$ and with at least one edge, a positive integer $m$, and an odd prime $p>h$. Let $S\subseteq\{0,1,\ldots,\ell\}^n$ be a finite set for which there are a $c\in\{0,1,\ldots,\ell\}^n$ and a constant $R$ such that $\norm{x-c}_2^2=R$ for every $x\in S$. Recall
\[
d_{\max}(S)=\max_{x,y\in S}\frac12\norm{x-y}_2^2
\]
and suppose
\begin{equation}\label{eq:prime-condition}
d_{\max}(S)<(m+1)p.
\end{equation}
Let $B\subseteq S$ contain no copy of $H$ whose edge lengths lie in $A_{m,p}$ from \labelcref{eq:amp}. Define
\begin{equation}\label{eq:graph-tensor}
T_H(x_1,\ldots,x_h)
=\prod_{ij\in E(H)}D_p(x_i,x_j)
=\prod_{ij\in E(H)}
\left(1-\left(\frac12\norm{x_i-x_j}_2^2\right)^{p-1}\right)
\end{equation}
as a function $B^h\to\F_p$. Also define
\begin{equation}\label{eq:diagonalized-tensor}
J_H=I_hT_H.
\end{equation}

\begin{lemma}\label{lem:semidiagonal}
The tensor $T_H$ is semi-diagonal on $B^h$. Moreover, $J_H$ is supported exactly on the full diagonal of $B^h$, where it has the nonzero value $(-1)^h(h-1)!$.
\end{lemma}

\begin{proof}
If $x_1=\cdots=x_h=x$, then every edge factor is $D_p(x,x)=1$, so
\[
T_H(x,\ldots,x)=1.
\]
Now suppose that $x_1,\ldots,x_h$ are all distinct. If every factor in \labelcref{eq:graph-tensor} were one, then \labelcref{eq:prime-condition} and the discussion following \labelcref{eq:detector} would imply
\[
\norm{x_i-x_j}_2\in A_{m,p}
\]
for every edge $ij\in E(H)$. The injective map $i\mapsto x_i$ would then be a copy of $H$ in $B$, contrary to the choice of $B$. At least one factor is therefore zero, so $T_H$ vanishes on every all-distinct tuple. This proves semidiagonality.

There are now three equality patterns to consider for $J_H=I_hT_H$. On the full diagonal, \Cref{lem:distinctness} gives $I_h=(-1)^h(h-1)!$, and $T_H=1$. This value is nonzero in $\F_p$ because $p>h$. On an all-distinct tuple, $I_h=1$ but $T_H=0$. On a tuple that is neither all distinct nor fully diagonal, $I_h=0$. Hence $J_H$ vanishes everywhere except on the full diagonal.
\end{proof}

The next two propositions upper-bound the size of $B$ in different ways. The first uses the semi-diagonal tensor $T_H$ and flattening rank. The second uses the diagonal tensor $J_H$ and partition rank.

\begin{proposition}[Flattening rank estimate]\label{prop:flattening}
Let the degree sequence of $H$ be
\[
d_1(H)\geq d_2(H)\geq\cdots\geq d_h(H).
\]
Then
\[
|B|\leq (h-1)|M_{\ell,n}((p-1)d_2(H))|.
\]
\end{proposition}

\begin{proof}
Relabel the vertices of $H$, if necessary, so that vertex $i$ has degree $d_i(H)$, and set $s=1$. By \Cref{lem:semidiagonal,cor:frank},
\begin{equation}\label{eq:frank-sum}
|B|\leq\sum_{i\neq s}\frank_i(T_H).
\end{equation}
We bound each term on the right.

Fix $i\in[h]$. Since $\norm{x-c}_2^2=R$ for every $x\in S$, \labelcref{eq:half-distance} rewrites the tensor as
\[
T_H(x_1,\ldots,x_h)=
\prod_{uv\in E(H)}
\left(1-( R-\ip{x_u-c}{x_v-c})^{p-1}\right).
\]
Only the $d_i(H)$ edge factors incident with vertex $i$ involve $x_i$. Each such factor has degree at most $p-1$ in the coordinates of $x_i$. Thus $T_H$, viewed as a polynomial in the coordinates of $x_i$ with all other variables held together, has total $x_i$-degree at most
\[
(p-1)d_i(H).
\]
By \Cref{lem:monomial-count}, it can be written on $B^h$ as
\[
T_H(x_1,\ldots,x_h)
=\sum_{\alpha\in M_{\ell,n}((p-1)d_i(H))}
x_i^\alpha g_\alpha((x_j)_{j\neq i}).
\]
Each summand has $i$-flattening rank one, and therefore
\begin{equation}\label{eq:individual-frank}
\frank_i(T_H)
\leq |M_{\ell,n}((p-1)d_i(H))|.
\end{equation}

For every $i\neq s$, we have $d_i(H)\leq d_2(H)$. The sets $M_{\ell,n}(L)$ are increasing with $L$, so \labelcref{eq:frank-sum,eq:individual-frank} give
\[
|B|
\leq\sum_{i\neq s}|M_{\ell,n}((p-1)d_i(H))|
\leq(h-1)|M_{\ell,n}((p-1)d_2(H))|.
\]
\end{proof}

The omission of the maximum degree vertex in this proof is the precise reason that $d_2(H)$, rather than the maximum degree $d_1(H)$, appears in the theorem. No ordering of the flattening ranks is needed; the omitted index is chosen from the graph before any ranks are estimated.

\begin{proposition}[Partition rank estimate]\label{prop:partition}
For every finite simple graph $H$,
\[
|B|\leq (2^h-2)|M_{\ell,n}((p-1)\lambda(H))|.
\]
\end{proposition}

\begin{proof}
By \Cref{lem:semidiagonal}, the restriction of $J_H$ to $B^h$ is a nonzero scalar multiple of the diagonal tensor. Hence \labelcref{eq:diagonal-rank} gives
\begin{equation}\label{eq:partition-lower}
|B|=\prank(J_H|_{B^h}).
\end{equation}
We construct a partition rank decomposition of $J_H$.

Expand $I_h$ using \labelcref{eq:ik}:
\begin{equation}\label{eq:jh-expansion}
J_H(x_1,\ldots,x_h)=
\sum_{\substack{\pi\in\Pi_{[h]}\\ \pi\neq\hat 1}}
\mu(\hat0,\pi)\delta_\pi(x_1,\ldots,x_h)T_H(x_1,\ldots,x_h).
\end{equation}
Fix one partition $\pi=P_1|\cdots|P_r$ in this sum. Since $\pi\neq\hat1$, we have $r\geq2$. Choose a representative $a_j=\min P_j$ from each block and set $y_j=x_{a_j}$. On the support of $\delta_\pi$, all variables in $P_j$ equal $y_j$.

If an edge of $H$ has both endpoints in $P_j$, its detector becomes $D_p(y_j,y_j)=1$. If several edges run between $P_j$ and $P_k$, their detectors are equal and their product is the same detector because $D_p$ takes only the values zero and one. Consequently,
\begin{equation}\label{eq:quotient-function}
T_H(x_1,\ldots,x_h)=
\prod_{jk\in E(H/\pi)}D_p(y_j,y_k)
\end{equation}
as a function on the support of $\delta_\pi$. This is the point at which the simple quotient graph, rather than a multigraph, enters the proof.

Choose a block $P_s$ corresponding to a minimum-degree vertex of $H/\pi$. By \labelcref{eq:half-distance}, the right-hand side of \labelcref{eq:quotient-function}, viewed as a function on $B^h$, can be represented by a polynomial in $y_s$ of degree at most
\[
(p-1)\delta(H/\pi)\leq(p-1)\lambda(H).
\]
By \Cref{lem:monomial-count}, the summand of \labelcref{eq:jh-expansion} belonging to $\pi$ can be written as a sum over at most
\[
|M_{\ell,n}((p-1)\lambda(H))|
\]
monomials in $y_s$. More explicitly, it has the form
\[
\sum_{\alpha\in M_{\ell,n}((p-1)\lambda(H))}
\left(\delta_{P_s}(x_{P_s})x_{a_s}^{\alpha}\right)
G_{\pi,\alpha}(x_{[h]\setminus P_s}),
\]
where $\delta_{P_s}$ enforces equality among the variables indexed by $P_s$, and all other equality conditions and polynomial factors are absorbed into $G_{\pi,\alpha}$. Every displayed term has partition rank one across the nontrivial split
\[
P_s\mid([h]\setminus P_s).
\]

For a fixed nonempty proper subset $P\subset[h]$ and a fixed monomial $x_{\min P}^{\alpha}$, sum the functions $G_{\pi,\alpha}$ over all partitions $\pi$ for which the selected minimum-degree block is $P$. This sum is still an arbitrary function of the complementary variables and therefore still gives one partition rank one term. There are $2^h-2$ choices for $P$ and at most $|M_{\ell,n}((p-1)\lambda(H))|$ choices for $\alpha$. Hence
\[
\prank(J_H|_{B^h})
\leq(2^h-2)|M_{\ell,n}((p-1)\lambda(H))|.
\]
Combining this with \labelcref{eq:partition-lower} proves the proposition.
\end{proof}

\section{From finite rank bounds to Euclidean colorings}\label{section:transfer}

We first handle a graph-theoretic point needed for disconnected $H$. If $sF$ denotes the disjoint union of $s$ copies of a graph $F$, then avoiding $sF$ has the same Euclidean chromatic cost as avoiding one copy of $F$. The relevant point is that an $sF$-free color class need not contain only finitely many copies of $F$, but its family of $F$-copies has matching number at most $s-1$.

\begin{proposition}\label{prop:deletion}
Let $F$ be a finite simple graph with at least one edge, let $A\subset\R_{>0}$ be finite, and let $Z\subset\R^n$ be finite. Then
\[
\chi_F(\R^n;A)=\chi_F(\R^n\setminus Z;A).
\]
Consequently, for every positive integer $s$,
\[
\chi_{sF}(\R^n;A)=\chi_F(\R^n;A).
\]
\end{proposition}

\begin{proof}
For $X\subseteq\R^n$, let $\cG_A^{(F)}(X)$ be the $|V(F)|$-uniform hypergraph whose vertices are the points of $X$ and whose hyperedges are the point sets that support copies of $F$ with edge lengths in $A$. A proper coloring of this hypergraph is exactly a coloring of $X$ with no monochromatic copy of $F$.

The relevant chromatic numbers are finite: a finite distance graph in $\R^n$ admits a periodic coloring by sufficiently small cubes, and a proper coloring of that distance graph avoids every graph $F$ with an edge. The inequality
\[
\chi_F(\R^n\setminus Z;A)\leq\chi_F(\R^n;A)
\]
follows by restriction. For the reverse inequality, the de Bruijn-Erd\H{o}s compactness theorem, in its standard hypergraph form, gives a finite set $W\subset\R^n$ such that
\[
\chi_F(W;A)=\chi_F(\R^n;A)
\]
\cite{debruijnerdos1951}. The set of translations $t$ for which $(W+t)\cap Z\neq\emptyset$ is finite, so choose any other $t\in\R^n$. Then $W+t\subset\R^n\setminus Z$, and translation preserves every distance. Therefore
\[
\chi_F(\R^n\setminus Z;A)
\geq\chi_F(W+t;A)
=\chi_F(W;A)
=\chi_F(\R^n;A).
\]

It remains to prove the assertion about $sF$. Since every copy of $sF$ contains a copy of $F$, every coloring that avoids $F$ also avoids $sF$. Hence
\[
\chi_{sF}(\R^n;A)\leq\chi_F(\R^n;A).
\]
For the reverse inequality, let $k=\chi_{sF}(\R^n;A)$ and choose a $k$-coloring with no monochromatic copy of $sF$. Fix one color. In the hypergraph of monochromatic copies of $F$, there are no $s$ pairwise vertex-disjoint hyperedges. Choose a maximal matching. It has at most $s-1$ hyperedges, and the union of its vertices meets every monochromatic copy of $F$ of that color; otherwise another disjoint hyperedge could be added. Doing this for each of the finitely many colors gives a finite set $Z$ meeting every monochromatic copy of $F$. The restricted coloring of $\R^n\setminus Z$ avoids $F$, so the first part of the proposition gives
\[
\chi_F(\R^n;A)
=\chi_F(\R^n\setminus Z;A)
\leq k
=\chi_{sF}(\R^n;A).
\]
\end{proof}

The following proposition isolates and simplifies the asymptotic transfer used in the proof of Naslund's multiple-distance clique bound \cite[Theorem~2]{naslund2022chromatic}. In Naslund's proof, the corresponding transfer is carried out through the constant composition and truncated theta-function arguments in \cite[Theorems~3 and~4]{naslund2022chromatic}.

\begin{proposition}\label{prop:asymptotic}
Fix a finite simple graph $H$ with at least one edge and positive integers $m$ and $q$. Suppose that there is a constant $C_H>0$, depending only on $H$, such that the following holds. For every $\ell,n\geq1$, every set $S\subseteq\{0,1,\ldots,\ell\}^n$ for which there are a $c\in\{0,1,\ldots,\ell\}^n$ and a constant $R$ such that
\[
\norm{x-c}_2^2=R
\]
for every $x\in S$, every prime $p>|V(H)|$ satisfying $d_{\max}(S)<(m+1)p$, and every $H$-free set $B\subseteq S$ in the distance graph with distance set $A_{m,p}$, one has
\[
|B|\leq C_H|M_{\ell,n}((p-1)q)|.
\]
Let $(\ell_n)$ be any sequence of positive integers and, for each $n$, write $\ell=\ell_n$ and let $S=S_n$ be such a set with $d_{\max}(S)=O(n)$. Then, for every fixed $t\in(0,1)$,
\begin{equation}\label{eq:theta-transfer}
\chi_H(\R^n;m)\geq
|S|\left(
\frac{t^{\frac{q d_{\max}(S)}{n(m+1)}}}
{1+t+\dots+t^\ell}
+o(1)
\right)^n.
\end{equation}
Consequently,
\begin{equation}\label{eq:gamma-transfer}
\chi_H(\R^n;m)\geq
\left(
\Gamma_{\chi}\sqrt{\frac{m+1}{q}}+o(1)
\right)^n.
\end{equation}
\end{proposition}

\begin{proof}
Fix $t\in(0,1)$, and write $\ell=\ell_n$ and $S=S_n$. Let $p$ be the smallest prime satisfying
\[
p>\max\left\{|V(H)|,\frac{d_{\max}(S)}{m+1}\right\}.
\]
This choice ensures that the assumed finite bound applies.

Suppose that $S$ is colored without a monochromatic copy of $H$ whose edge lengths lie in $A_{m,p}$. Every color class is then an $H$-free subset of $S$ and consequently has size at most $C_H|M_{\ell,n}((p-1)q)|$. Hence the number of colors is at least
\begin{equation}\label{eq:finite-color-lower}
\frac{|S|}{C_H|M_{\ell,n}((p-1)q)|}.
\end{equation}

We next estimate the number of monomials in the denominator. If $\alpha\in M_{\ell,n}((p-1)q)$, then $|\alpha|\leq(p-1)q$, and therefore $t^{|\alpha|-(p-1)q}\geq1$. Summing over these vectors and then enlarging the sum to the entire box $\{0,1,\ldots,\ell\}^n$ gives
\begin{equation}\label{eq:monomial-gf}
|M_{\ell,n}((p-1)q)|
\leq
\frac{(1+t+\cdots+t^\ell)^n}{t^{(p-1)q}}.
\end{equation}

Dilating $S$ by $(2p)^{-1/2}$ changes $A_{m,p}$ into the fixed distance set $A_m=\{1,\sqrt2,\ldots,\sqrt m\}$. Thus the resulting lower bound holds already for $\chi_H(\R^n;A_m)$, and hence also for $\chi_H(\R^n;m)$. Combining \labelcref{eq:finite-color-lower,eq:monomial-gf}, we obtain
\begin{equation}\label{eq:color-before-composition}
\chi_H(\R^n;m)
\geq
\frac{|S|}{C_H}
\frac{t^{(p-1)q}}{(1+t+\cdots+t^\ell)^n}.
\end{equation}

It remains to compare the chosen prime with $d_{\max}(S)/(m+1)$. Set $x=d_{\max}(S)/(m+1)$. For sufficiently large $x$, the prime-gap theorem of Baker, Harman, and Pintz gives a prime above $x$ of size at most $x+O(x^{0.525})$ \cite{baker2001}. When $x$ is bounded, the additional requirement $p>|V(H)|$ affects the estimate by at most a constant depending on the fixed graph $H$. Consequently,
\[
p\leq
\frac{d_{\max}(S)}{m+1}
+O\!\left(d_{\max}(S)^{0.525}\right)
+O_H(1).
\]
Since $d_{\max}(S)=O(n)$, there is a nonnegative quantity $\eta_n=o(n)$ such that $p-1\leq d_{\max}(S)/(m+1)+\eta_n$. Because $0<t<1$, this implies
\[
t^{(p-1)q}
\geq
t^{q d_{\max}(S)/(m+1)}t^{q\eta_n}.
\]
Since $(C_H^{-1}t^{q\eta_n})^{1/n}=1+o(1)$, substituting this estimate into \labelcref{eq:color-before-composition} proves \labelcref{eq:theta-transfer}.

To prove \labelcref{eq:gamma-transfer}, fix a constant $a>0$ and choose
\[
d_n=\left\lfloor\sqrt{\frac{2an(m+1)}{q}}\right\rfloor
\qquad\text{and}\qquad
\ell=2d_n.
\]
Now let $S$ be the set supplied by \Cref{lem:size_of_S}. This set satisfies the constant norm condition in the hypothesis, and its diameter bound gives $d_{\max}(S)\leq d_n^2/2\leq an(m+1)/q$. Moreover, \Cref{lem:size_of_S} gives $|S|^{1/n}\geq d_n\sqrt{e\pi/(2n)}(1+o(1))=\sqrt{e\pi a(m+1)/q}+o(1)$.

For every fixed $t\in(0,1)$, we have $t^{q d_{\max}(S)/(n(m+1))}\geq t^a$ and
$\frac{1}{1+t+\dots+t^\ell}=\frac{1-t}{1-t^{\ell+1}}\geq1-t$. Since $a$ and $t$ were arbitrary, \labelcref{eq:theta-transfer} therefore gives
\begin{align*}
\chi_H(\R^n;m)
&\geq
\left(
\sqrt{\frac{m+1}{q}}\cdot\sqrt{e\pi}\cdot
\max_{\substack{0<t<1\\a>0}}
t^a(1-t)\sqrt a
+o(1)
\right)^n\\
&=
\left(
\Gamma_{\chi}\sqrt{\frac{m+1}{q}}+o(1)
\right)^n.
\end{align*}
Indeed, writing $t=e^{-x}$, the maximum over $a$ is attained at $a=(2x)^{-1}$, and the last equality is exactly the definition of $\Gamma_{\chi}$. This completes the proof.
\end{proof}

We now assemble the proof of the main theorem.

\begin{proof}[Proof of \Cref{thm:main}]
Apply \Cref{prop:asymptotic} to the flattening rank estimate in \Cref{prop:flattening}. The constant $C_H=h-1$ is independent of $n$, and we obtain
\begin{equation}\label{eq:d2-main-bound}
\chi_H(\R^n;m)
\geq
\left(\Gamma_{\chi}\sqrt{\frac{m+1}{d_2(H)}}+o(1)\right)^n.
\end{equation}

For the partition rank estimate, let $C_1,\ldots,C_s$ be the connected components of $H$, and choose a connected envelope $F$ for which
\[
\lambda(F)=\hlam(H).
\]
Such an envelope exists because the possible values of $\lambda(F)$ are positive integers and the family of envelopes is nonempty. By \Cref{prop:partition,prop:asymptotic},
\[
\chi_F(\R^n;m)
\geq
\left(\Gamma_{\chi}\sqrt{\frac{m+1}{\lambda(F)}}+o(1)\right)^n.
\]
The graph $sF$ contains a copy of $H$: embed $C_i$ into the $i$th copy of $F$. Therefore, for every distance set $A$,
\[
\chi_H(\R^n;A)\geq\chi_{sF}(\R^n;A).
\]
By \Cref{prop:deletion}, the right-hand side equals $\chi_F(\R^n;A)$. Taking the maximum over all $m$-element distance sets gives
\begin{equation}\label{eq:hlam-main-bound}
\chi_H(\R^n;m)
\geq
\left(\Gamma_{\chi}\sqrt{\frac{m+1}{\hlam(H)}}+o(1)\right)^n.
\end{equation}

The better of \labelcref{eq:d2-main-bound,eq:hlam-main-bound} is the estimate with the smaller denominator. By definition, that denominator is
\[
\min\{d_2(H),\hlam(H)\}=\Xi(H).
\]
\end{proof}

\section{The gain from the statistic \texorpdfstring{$\Xi$}{Xi}}
\label{section:the-statistic}

The main purpose of the statistic $\Xi$ is to retain information about the structure of the forbidden graph that is lost under a direct application of Naslund's theorem. Let $H$ have $h$ vertices. Since $H$ is a subgraph of $K_h$, every monochromatic copy of $K_h$ contains a monochromatic copy of $H$. Thus Naslund's theorem for $K_h$ can always be used to obtain a lower bound for $\chi_H(\R^n;m)$. The resulting bound and the graph sensitive bound of \Cref{thm:main} are
\begin{equation}\label{eq:xi-versus-clique}
\begin{aligned}
\chi_H(\R^n;m)
&\geq
\left(\Gamma_{\chi}\sqrt{\frac{m+1}{h-1}}+o(1)\right)^n
&&\text{from the clique reduction},\\
\chi_H(\R^n;m)
&\geq
\left(\Gamma_{\chi}\sqrt{\frac{m+1}{\Xi(H)}}+o(1)\right)^n
&&\text{from \Cref{thm:main}.}
\end{aligned}
\end{equation}
The first estimate treats $H$ as though every pair of its vertices were joined by an edge. The second estimate uses the edge structure that is actually present.

Ignoring the common $o(1)$ term, replacing $h-1$ by $\Xi(H)$ multiplies the exponential base by $\sqrt{(h-1)/\Xi(H)}$. Thus every strict inequality $\Xi(H)<h-1$ improves upon the complete-graph reduction. This difference is particularly important because the improvement occurs inside a quantity raised to the $n$th power. Even a constant-factor improvement in the base becomes an exponential-factor improvement as the Euclidean dimension grows.

For each fixed graph $H$, the asymptotic statement in \Cref{thm:main} concerns $n\to\infty$. In this section we also allow $H$ to range through a family of growing graphs in order to understand how much is gained by using $\Xi(H)$ instead of $|V(H)|-1$. We first establish general bounds on $\Xi$, then compute it for several representative families.

\subsection{Estimating the improvement}

We begin with a monotonicity fact that justifies the definition of $\hlam$ and allows us to compare a connected component with any connected envelope containing it.

\begin{proposition}\label{prop:lambda-monotone}
If a connected graph $G$ is a subgraph of a connected graph $F$, then $\lambda(G)\leq\lambda(F)$. Consequently, if $H$ has connected components $C_1,\ldots,C_s$, then $\hlam(H)\geq\max_{1\leq i\leq s}\lambda(C_i)$. If $H$ is connected, then $\hlam(H)=\lambda(H)$.
\end{proposition}

\begin{proof}
If $G=K_1$, the assertion is immediate, so suppose that $|V(G)|\geq2$. Fix a partition $\pi$ of $V(G)$ into at least two blocks. Extend it to a partition $\widetilde\pi$ of $V(F)$ by placing every vertex of $F\setminus G$ into any existing block. The quotient $G/\pi$ is a spanning subgraph of $F/\widetilde\pi$: every adjacency produced by an edge of $G$ is still present, and extra edges of $F$ can only add adjacencies. Hence $\delta(G/\pi)\leq\delta(F/\widetilde\pi)\leq\lambda(F)$. Maximizing over $\pi$ gives $\lambda(G)\leq\lambda(F)$.

Every connected envelope of $H$ contains each $C_i$, so the first consequence follows by minimizing over envelopes. If $H$ is connected, then $H$ itself is an admissible envelope, giving $\hlam(H)\leq\lambda(H)$. Monotonicity gives the reverse inequality for every connected envelope.
\end{proof}

The next estimate shows that a graph with relatively few edges cannot have a large value of $\Xi$. This already gives a broad class of graphs for which \Cref{thm:main} improves upon the denominator $h-1$.

\begin{proposition}\label{prop:edge-bound}
If a graph $G$ has $e\geq1$ edges, then
\[
\lambda(G)\leq
\left\lfloor\frac{\sqrt{8e+1}-1}{2}\right\rfloor.
\]
If $H$ has $e\geq1$ edges, then $\Xi(H)\leq\hlam(H)<2\sqrt e$. In particular, $\Xi(H)=O(\sqrt{|E(H)|})$.
\end{proposition}

\begin{proof}
Let $Q=G/\pi$ and put $d=\delta(Q)$. The quotient has at least $d+1$ vertices, so $|E(Q)|\geq d(d+1)/2$. Passing to a simple quotient cannot increase the number of edges, and therefore $d(d+1)/2\leq |E(Q)|\leq e$. Solving this quadratic inequality for $d$ and then maximizing over $\pi$ proves the first assertion.

For the second assertion, discard the isolated components of $H$ when constructing an envelope. This causes no problem because a single vertex is already a subgraph of every nonempty connected graph. Suppose the remaining nontrivial components number $r$. Since each such component contains an edge, we have $r\leq e$. Take disjoint copies of these components and add $r-1$ edges to connect them. The resulting envelope $F$ has at most $e+r-1\leq2e-1$ edges. The first part gives $\hlam(H)\leq\lambda(F)<2\sqrt e$, and $\Xi(H)\leq\hlam(H)$ by definition.
\end{proof}

Suppose that $H_h$ is a graph on $h$ vertices with at least one edge and $|E(H_h)|=o(h^2)$. The preceding proposition gives $\Xi(H_h)=o(h)$, so the gain $\sqrt{(h-1)/\Xi(H_h)}$ over the clique reduction tends to infinity. In particular, if $|E(H_h)|=O(h)$, then the proposition alone gives $\Xi(H_h)=O(\sqrt h)$ and hence a gain of at least order $h^{1/4}$ in the exponential base. The second largest degree term in the definition of $\Xi$ can produce an even larger improvement.

\subsection{Representative graph families}

Recall that the join $G_1+G_2$ of disjoint graphs is obtained from $G_1\sqcup G_2$ by adding every edge between $V(G_1)$ and $V(G_2)$. For $1\leq s<h$, put $J_{s,h}=K_s+\overline{K_{h-s}}$. This is a split graph consisting of a clique of $s$ universal vertices and an independent set of $h-s$ vertices.

The final column of the following table records the factor by which the leading exponential base in \Cref{thm:main} improves upon the base obtained from the clique reduction.

\begin{proposition}\label{prop:family-table}
For the indicated ranges, the following values hold.
\begin{center}
\small
\begin{tabular}{lcccc}
\toprule
Graph family & $\hlam$ & $d_2$ & $\Xi$ & Gain in base\\
\midrule
$K_h$, $h\geq2$ & $h-1$ & $h-1$ & $h-1$ & $1$\\
$P_h$, $h\geq4$ & $\Theta(\sqrt h)$ & $2$ & $2$ & $\sqrt{(h-1)/2}$\\
$C_h$, $h\geq3$ & $\Theta(\sqrt h)$ & $2$ & $2$ & $\sqrt{(h-1)/2}$\\
$K_{1,h-1}$, $h\geq2$ & $1$ & $1$ & $1$ & $\sqrt{h-1}$\\
$rK_2$, $r\geq1$ & $1$ & $1$ & $1$ & $\sqrt{2r-1}$\\
$J_{s,h}$, $2\leq s<h$ & $s$ & $h-1$ & $s$ & $\sqrt{(h-1)/s}$\\
\bottomrule
\end{tabular}
\end{center}
\end{proposition}

\begin{proof}
For $K_h$, every quotient is a complete graph on at most $h$ vertices, and the discrete partition gives minimum degree $h-1$. Thus $\lambda(K_h)=h-1$. The graph is connected and regular, so $\hlam(K_h)=d_2(K_h)=\Xi(K_h)=h-1$. For paths and cycles, \Cref{prop:edge-bound} gives $\lambda=O(\sqrt h)$. For the matching lower bound, choose the largest odd integer $q=2k+1$ such that $|E(K_q)|=k(2k+1)\leq h-1$ in the path case, or at most $h$ in the cycle case. Every vertex of $K_{2k+1}$ has even degree, so this complete graph has an Euler circuit. Reading the vertices along the circuit gives a walk whose support contains every edge of $K_{2k+1}$. If the path or cycle has extra steps, pad the walk by staying at one vertex. Such a stationary step corresponds to an edge whose endpoints lie in the same quotient block, so it disappears in the simple quotient. This constructs a quotient isomorphic to $K_{2k+1}$ and proves that $\lambda(P_h),\lambda(C_h)=\Omega(\sqrt h)$. Both graphs are connected. For $P_h$ with $h\geq4$, the second largest degree is two, and every cycle is $2$-regular. Hence $\Xi(P_h)=\Xi(C_h)=2$. Every quotient of a star is again a star or a single edge. The block containing the center is adjacent to every block containing only leaves, while two leaf-only blocks are not adjacent. Hence $\lambda(K_{1,h-1})=1$. Its second largest degree is also one, so $\Xi(K_{1,h-1})=1$. For the matching $rK_2$, the connected graph $K_2$ is an envelope for every component, so $\hlam(rK_2)\leq1$. Every connected graph with an edge has quotient parameter at least one, giving equality. All nonisolated degrees in the matching are one, and therefore $d_2(rK_2)=\Xi(rK_2)=1$. Finally, consider $J_{s,h}$ with $2\leq s<h$. Its discrete quotient has minimum degree $s$, so $\lambda(J_{s,h})\geq s$. In an arbitrary quotient, call a block a core block if it contains at least one of the $s$ clique vertices. There are at most $s$ core blocks. Every block containing only independent-set vertices is adjacent exactly to the core blocks. If such a block exists, the quotient minimum degree is at most $s$. If no such block exists, every block is a core block and the quotient is a clique on at most $s$ vertices, whose minimum degree is at most $s-1$. Thus $\lambda(J_{s,h})=s$. The graph is connected, so $\hlam(J_{s,h})=s$. Since the $s$ clique vertices are universal and $s\geq2$, we have $d_2(J_{s,h})=h-1$. Therefore $\Xi(J_{s,h})=s$.
\end{proof}

Equality with the clique denominator occurs precisely in the complete-graph cases represented in the table. For a path or cycle, the clique reduction pays $h-1$, while \Cref{thm:main} pays only $2$. The graph sensitive base is therefore larger by a factor of $\sqrt{(h-1)/2}$. In particular, the denominator in the new bound remains fixed even as the path or cycle grows. Stars give an even cleaner comparison. Naslund's clique reduction pays $h-1$, but $\Xi(K_{1,h-1})=1$, so the gain in the exponential base is $\sqrt{h-1}$. This happens because only the center has large degree. The flattening argument omits that vertex and pays the second largest degree, which is one. The split graphs display the complementary partition rank mechanism. When $s$ is fixed, the graph $J_{s,h}$ has at least two universal vertices and therefore $d_2(J_{s,h})=h-1$. The flattening estimate alone gives no improvement over the clique denominator. Nevertheless, the quotient parameter satisfies $\hlam(J_{s,h})=s$, so $\Xi(J_{s,h})=s$. Partition rank therefore replaces the growing denominator $h-1$ by the fixed number $s$, producing a gain of order $\sqrt h$ in the exponential base.

\subsection{When \texorpdfstring{$\Xi$}{Xi} stays constant}

The largest separation from the clique reduction occurs when the order of the graph tends to infinity while $\Xi$ remains bounded. In that situation, the denominator in \Cref{thm:main} stays constant, while the clique denominator $h-1$ tends to infinity. The gain in the exponential base is therefore of order $\sqrt h$.

There are two independent ways for a family to have bounded $\Xi$. The flattening mechanism gives the immediate inequality $\Xi(H)\leq d_2(H)$. Thus bounded maximum degree, and more generally bounded second largest degree, implies bounded $\Xi$. This includes paths, cycles, bounded-degree trees, fixed-degree regular expanders, and subdivisions of any fixed graph.

The partition mechanism can keep $\Xi$ bounded even when $d_2$ grows linearly. The split graphs $J_{s,h}$ with fixed $s$ are the cleanest example: when $s\geq2$, they have two universal vertices and hence $d_2$ is $h-1$, but $\Xi$ is $s$. Repeated connected components provide another exact constant family.

\begin{proposition}\label{prop:constant-families}
Let $F$ be a fixed connected graph with at least one edge.
\begin{enumerate}[label=\textup{(\roman*)}]
\item For every $r\geq1$, one has $\hlam(rF)=\lambda(F)$. For $r\geq2$, one has $\Xi(rF)=\min\{\lambda(F),\Delta(F)\}$, which is independent of $r$.
\item If $H$ is a $d$-regular graph with $d\geq1$, not necessarily connected, then $\Xi(H)=d$.
\end{enumerate}
\end{proposition}

\begin{proof}
The graph $F$ itself is a connected envelope for $rF$, so $\hlam(rF)\leq\lambda(F)$. Conversely, every connected envelope of $rF$ contains a copy of $F$, and \Cref{prop:lambda-monotone} gives the reverse inequality. Thus $\hlam(rF)=\lambda(F)$. If $r\geq2$, two different copies of $F$ contain vertices of degree $\Delta(F)$, so the second largest degree of $rF$ is $\Delta(F)$. This proves part (i). If $H$ is $d$-regular, then $d_2(H)=d$. Every connected component $C_i$ has minimum degree $d$, so its discrete quotient shows that $\lambda(C_i)\geq d$. By \Cref{prop:lambda-monotone}, $\hlam(H)\geq d$. It follows that $d\leq\Xi(H)\leq d_2(H)=d$, proving part (ii).
\end{proof}

Let $f=|V(F)|$ and let $c_F=\min\{\lambda(F),\Delta(F)\}$. The graph $rF$ has $rf$ vertices, but $\Xi(rF)=c_F$ for every $r\geq2$. The clique reduction therefore pays $rf-1$, while the graph sensitive theorem pays the fixed number $c_F$. The gain in the exponential base is $\sqrt{(rf-1)/c_F}=\Theta(\sqrt r)$. The regular-graph statement gives another useful family. If $H_h$ is a $d$-regular graph on $h$ vertices and $d$ remains fixed, then $\Xi(H_h)=d$, so the gain over the clique reduction is $\sqrt{(h-1)/d}=\Theta(\sqrt h)$. In particular, fixed-degree regular expanders have constant $\Xi$ despite their strong global connectivity. When the regular degree grows with the graph, the gain may be smaller but still nontrivial. For example, $\Xi(K_{r,r})=r$ while $|V(K_{r,r})|=2r$, so the gain over the clique reduction tends to $\sqrt2$. At the opposite endpoint, complete graphs satisfy $\Xi(K_h)=h-1$, and no improvement is possible because the original graph is already the clique used in Naslund's theorem.

\subsection{Intermediate scales}

The improvement supplied by $\Xi$ is not restricted to the two extremes of a bounded denominator and the clique denominator $h-1$. The split graphs realize every intermediate order of growth. If $s=s(h)$ satisfies $1\leq s(h)<h$, then $\Xi(J_{s(h),h})=s(h)$. The clique reduction pays $h-1$, while the graph sensitive theorem pays $s(h)$, so the gain in the exponential base is $\sqrt{(h-1)/s(h)}$.

For example, choosing $s(h)=\Theta(\log h)$ gives a gain of order $\sqrt{h/\log h}$. Choosing $s(h)=\Theta(h^{1/3})$ gives a gain of order $h^{1/3}$, while choosing $s(h)=\Theta(\sqrt h)$ gives a gain of order $h^{1/4}$. Thus $\Xi$ can produce a substantial improvement even when it is not bounded. More generally, suppose that $\Xi(H_h)=\Theta(h^\alpha)$ for some $0<\alpha<1$. The clique reduction contributes a factor of order $h^{-1/2}$ to the exponential base, whereas \Cref{thm:main} contributes a factor of order $h^{-\alpha/2}$. Their ratio is $h^{(1-\alpha)/2}$, which tends to infinity.

The role of $\Xi(H)$ is therefore quantitative rather than merely notational. The denominator $h-1$ treats every graph on $h$ vertices as though every pair of vertices were constrained. The denominator $\Xi(H)$ records the cheapest separation made available by the actual edges of $H$. Depending on the graph family, this can improve the exponential base by a constant factor, by a power of $h$, or by the full order $\sqrt h$ that occurs when $\Xi$ remains bounded.

\section{Concluding remarks}\label{section:conclusion}

Theorem~\ref{thm:main} should be viewed as a graph sensitive extension of Naslund's multiple distance theorem. The new issue is that an arbitrary graph has two distinct kinds of inexpensive variable separation. Equality patterns turn the graph into simple quotients and lead to $\hlam(H)$, while semi-diagonal weighted flattening rank permits the single most expensive vertex to be discarded and leads to $d_2(H)$. This begs the question: are there ways to optimize this further by changing other parts of Naslund's work? This includes changing the congruence detector, the lattice-sphere construction, and the final constant $\Gamma_{\chi}$. 

The quotient parameter also records a genuine limitation of the direct partition rank tensor. A graph with few edges can have a dense quotient because a long walk may trace every edge of a much denser graph. Paths and cycles make this visible. Any improvement that reduces the quotient cost while using the same distinctness-indicator expansion would have to exploit cancellation between different partition lattice terms, rather than bounding each selected block separately. The partition-indicator framework of \cite{omar2025partition} may be useful for organizing such cancellations.

The weighted flattening theorem is independent of Euclidean geometry. Its corollary
\[
\sum_{i\neq s}\frank_i(T)\geq|B|
\]
shows that a semi-diagonal tensor remains strongly constrained even after any one flattening rank is removed. In applications where one variable has an unusually large polynomial degree, this can replace a maximum degree parameter by a second largest degree parameter. It would be interesting to identify other extremal problems in which that distinction improves bounds.

Finally, there are also geometric limitations. The main theorem concerns noninduced graph copies. For an induced copy, every nonedge would impose an inequality saying that a distance must not belong to the allowed set. A naive polynomial factor for such an inequality vanishes on the full diagonal and destroys the diagonal rank lower bound. The induced graph-copy problems studied by Axenovich, Liu, and Sagdeev therefore require information not present in the tensor used here \cite{axenovich2025ramsey}.

\section*{Acknowledgments}

The authors thank the Fields Institute FUSRP for support throughout the project. Mohamed Omar was partially supported by research funds from York University and NSERC Discovery Grant \#RGPIN-2025-06304.


\begin{thebibliography}{99}

\bibitem{axenovich2025ramsey}
M.~Axenovich, D.~Liu, and A.~Sagdeev,
\newblock Ramsey problems for graphs in Euclidean spaces and Cartesian powers,
\newblock \emph{arXiv preprint} arXiv:2512.15516, 2025.

\bibitem{baker2001}
R.~C. Baker, G.~Harman, and J.~Pintz,
\newblock The difference between consecutive primes. II,
\newblock \emph{Proc. London Math. Soc.} \textbf{83} (2001), 532-562.

\bibitem{capset}
J.~S. Ellenberg and D.~Gijswijt,
\newblock On large subsets of $\mathbb F_q^n$ with no three-term arithmetic progression,
\newblock \emph{Ann. of Math.} \textbf{185} (2017), 339-343.

\bibitem{correia2021flattening}
D.~Munh\'a Correia, B.~Sudakov, and I.~Tomon,
\newblock Flattening rank and its combinatorial applications,
\newblock \emph{Linear Algebra Appl.} \textbf{625} (2021), 113-125.

\bibitem{crootlevpach}
E.~Croot, V.~F. Lev, and P.~P. Pach,
\newblock Progression-free sets in $\mathbb Z_4^n$ are exponentially small,
\newblock \emph{Ann. of Math.} \textbf{185} (2017), 331-337.

\bibitem{debruijnerdos1951}
N.~G. de Bruijn and P.~Erd\H{o}s,
\newblock A colour problem for infinite graphs and a problem in the theory of relations,
\newblock \emph{Nederl. Akad. Wetensch. Proc. Ser. A} \textbf{54} (1951), 371-373.

\bibitem{degrey2018}
A.~D.~N.~J. de Grey,
\newblock The chromatic number of the plane is at least $5$,
\newblock \emph{Geombinatorics} \textbf{28} (2018), 18-31.

\bibitem{erdos1973i}
P.~Erd\H{o}s, R.~L. Graham, P.~Montgomery, B.~L. Rothschild, J.~Spencer, and E.~G. Straus,
\newblock Euclidean Ramsey theorems. I,
\newblock \emph{J. Combin. Theory Ser. A} \textbf{14} (1973), 341-363.

\bibitem{erdos1975ii}
P.~Erd\H{o}s, R.~L. Graham, P.~Montgomery, B.~L. Rothschild, J.~Spencer, and E.~G. Straus,
\newblock Euclidean Ramsey theorems. II,
\newblock in \emph{Infinite and Finite Sets}, Colloq. Math. Soc. J\'anos Bolyai, vol.~10,
North-Holland, 1975, pp.~529-557.

\bibitem{erdos1975iii}
P.~Erd\H{o}s, R.~L. Graham, P.~Montgomery, B.~L. Rothschild, J.~Spencer, and E.~G. Straus,
\newblock Euclidean Ramsey theorems. III,
\newblock in \emph{Infinite and Finite Sets}, Colloq. Math. Soc. J\'anos Bolyai, vol.~10,
North-Holland, 1975, pp.~559-583.

\bibitem{exoo2020}
G.~Exoo and D.~Ismailescu,
\newblock The chromatic number of the plane is at least $5$: a new proof,
\newblock \emph{Discrete Comput. Geom.} \textbf{64} (2020), 216-226.

\bibitem{franklrod1990}
P.~Frankl and V.~R\"odl,
\newblock A partition property of simplices in Euclidean space,
\newblock \emph{J. Amer. Math. Soc.} \textbf{3} (1990), 1-7.

\bibitem{franklwilson1981}
P.~Frankl and R.~M. Wilson,
\newblock Intersection theorems with geometric consequences,
\newblock \emph{Combinatorica} \textbf{1} (1981), 357-368.

\bibitem{graham2004}
R.~L. Graham,
\newblock Euclidean Ramsey theory,
\newblock in \emph{Handbook of Discrete and Computational Geometry}, second ed.,
CRC Press, Boca Raton, FL, 2004.

\bibitem{kupavskii2010chromatic}
A.~B. Kupavskii,
\newblock The chromatic number of the space $\mathbb R^n$ with a set of forbidden distances,
\newblock \emph{Dokl. Math.} \textbf{82} (2010), 963-966.

\bibitem{larmanrogers1972}
D.~G. Larman and C.~A. Rogers,
\newblock The realization of distances within sets in Euclidean space,
\newblock \emph{Mathematika} \textbf{19} (1972), 1-24.

\bibitem{naslund2020partition}
E.~Naslund,
\newblock The partition rank of a tensor and $k$-right corners in $\mathbb F_q^n$,
\newblock \emph{J. Combin. Theory Ser. A} \textbf{174} (2020), 105190.

\bibitem{naslund2022chromatic}
E.~Naslund,
\newblock The chromatic number of $\mathbb R^n$ with multiple forbidden distances,
\newblock \emph{Mathematika} \textbf{69} (2023), 692-718.

\bibitem{omar2025partition}
M.~Omar,
\newblock Partition rank and partition lattices,
\newblock \emph{Order} \textbf{42} (2025), 371-388.

\bibitem{raigorodskii2000}
A.~M. Raigorodskii,
\newblock On the chromatic number of a space with forbidden distances,
\newblock \emph{Russian Math. Surveys} \textbf{55} (2000), 351-352.

\bibitem{soifer2009}
A.~Soifer,
\newblock \emph{The Mathematical Coloring Book: Mathematics of Coloring and the Colorful Life of Its Creators},
\newblock Springer, New York, 2009.

\bibitem{stanley2012}
R.~P. Stanley,
\newblock \emph{Enumerative Combinatorics}, vol.~1, second ed.,
\newblock Cambridge University Press, Cambridge, 2012.

\bibitem{tao2016capset}
T.~Tao,
\newblock A symmetric formulation of the Croot-Lev-Pach-Ellenberg-Gijswijt capset bound,
\newblock blog post, 2016,
\newblock \href{https://terrytao.wordpress.com/2016/05/18/a-symmetric-formulation-of-the-croot-lev-pach-ellenberg-gijswijt-capset-bound/}{terrytao.wordpress.com}.

\end{thebibliography}
\end{document}